\documentclass[11pt]{amsart}

\usepackage{graphicx}
\usepackage{color}
\usepackage{enumerate,url,amssymb,mathrsfs,upref,pdfsync}

\newtheorem{theorem}{Theorem}[section]
\newtheorem{lemma}[theorem]{Lemma}

\newtheorem{proposition}[theorem]{Proposition}

\newtheorem*{proposition*}{Proposition}

\theoremstyle{definition}
\newtheorem{definition}[theorem]{Definition}
\newtheorem{example}[theorem]{Example}

\theoremstyle{remark}
\newtheorem{remark}[theorem]{Remark}

\numberwithin{equation}{section}

\newcommand{\abs}[1]{\lvert#1\rvert}

\newcommand{\A}{\mathbb{A}}

\newcommand{\C}{\mathbb{C}}

\newcommand{\E}{\mathcal{E}}
\newcommand{\EE}{\mathsf{E}}

\newcommand{\R}{\mathbb{R}}

\newcommand{\X}{\mathbb{X}}
\newcommand{\Y}{\mathbb{Y}}

\newcommand{\dtext}{\textnormal d}
\newcommand{\D}{\mathfrak{D}}

\newcommand{\Ho}{\mathsf{H}}

\newcommand{\cto}{\xrightarrow[]{{}_{c\delta}}}
\newcommand{\onto}{\overset{{}_{\textnormal{\tiny{onto}}}}{\longrightarrow}}
\DeclareMathOperator{\dist}{dist} 
 
\DeclareMathOperator{\id}{id}

\def\XXint#1#2#3{{\setbox0=\hbox{$#1{#2#3}{\int}$}
\vcenter{\hbox{$#2#3$}}\kern-.5\wd0}}

\def\le{\leqslant}
\def\ge{\geqslant}

\begin{document}

\title[{Energy-minimal mappings between Riemann surfaces}]{Lipschitz regularity of energy-minimal mappings between  doubly connected Riemann surfaces}

\author{David Kalaj}
\address{Faculty of Natural Sciences and
Mathematics, University of Montenegro, Cetinjski put b.b. 81000
Podgorica, Montenegro} \email{davidk@ac.me}




\begin{abstract} Let  $M$ and $N$ be doubly connected
 Riemann surfaces with $\mathscr{C}^{1,\alpha}$ boundaries and with nonvanishing conformal metrics $\sigma$ and $\wp$ respectively, and assume that $\wp$ is a smooth
metric with bounded Gauss curvature $\mathcal{K}$ and finite area.
Assume that ${\Ho}_\rho(M, N)$ is the class of all $\mathscr{W}^{1,2}$ bomeomorphisms between $M$ and $N$ and assume that $\mathcal{E}^\wp: \overline{\Ho}_\rho(M, N)\to \mathbf{R}$ is the Dirichlet-energy functional, where $\overline{\Ho}_\rho(M, N)$ is the closure of ${\Ho}_\rho(M, N)$ in $\mathscr{W}^{1,2}(M,N)$. By using a result of Iwaniec, Kovalev and Onninen in \cite{duke} that the minimizer,  is locally Lipschitz, we prove that the minimizer,  of the energy functional $\mathcal{E}^\wp$, which is not a diffeomorphism in general, is a globally Lipschitz mapping of $M$ onto $N$.
\end{abstract}

\maketitle \tableofcontents
\section{Introduction}\label{intsec}
The primary goal of this paper is to study Lipschitz behaviour of stationary deformations of the Dirichlet  energy
 of mappings between doubly-connected Riemann surfaces. The existence part has been proved. More precisely, Iwaniec, Koch, Kovalev and Onninen in \cite{inve} proved that there exists so-called deformation $f$ that maps a double connected domain $\X$ onto a doubly connected domain $\Y$ in the complex plane and which minimizes the Diriclet energy integral throughout the class os deformations $\mathcal{D}(\X,\Y)$ which contains the class of all Sobolev $\mathscr{W}^{1,2}(\X)$ homeomorphism. The minimizer is a harmonic diffeomorphism, provided that the conformal moduli satisfies the relation $\mathrm{Mod}(\X)\le \mathrm{Mod}(\Y)$. Moreover, the Dirichlet energy is invariant under conformal change of the original domain. This is why the original domain can be chosen to be equal to  $\X=\A(r,R):=\{z: r<|z|<R\}$.  The Hopf's differential defined by $\mathrm{Hopf}(f)=f_z\overline{f}_z$ has very special form for so-called stationary deformations namely $$\mathrm{Hopf}(f)(z) = \frac{c}{z^2}, \ \ z\in \A(r,R).$$
Later this result has been generalized by the author in \cite{calculus} for so-called admissible metrics $\wp$ in $\Y$.

In this case Hopf's differential is defined by $$\mathrm{Hopf}(f)=\wp^2(f(z))f_z\overline{f}_z$$ and has a very special form for a so-called stationary deformation namely $$\mathrm{Hopf}(f)(z) = \frac{c}{z^2}, \ \ z\in \A(r,R).$$

Further in \cite{duke}, Iwaniec, Kovalev  and Onninen proved that every stationary deformation is locally  Lipschitz in the domain.

In \cite{geom} the author proved that if $f:\X\onto \Y$ is a $\rho-$harmonic diffeomorphic minimizer, and  $\partial \X,\partial\Y\in \mathscr{C}^{2}$ then $f$ is Lipschitz continuous up to the boundary.

In \cite{annalen} the author and Lamel proved that a minimizer of the Euclidean Dirichlet energy which is a diffeomophic surjection of $\X$ onto $\Y$ has smooth extension up to the boundary. More precisely they proved that, they proved that, if $f:\X\onto \Y$ is a Euclidean diffeomorphic minimizer, so that $\partial \X,\partial\Y\in \mathscr{C}^{1,\alpha}$, then the $f\in\mathscr{C}^{1,\alpha'}(\overline{\X})$, where $\alpha'=\alpha$, if $\mathrm{Mod}(\X)\ge \mathrm{Mod}(\Y)$ and $\alpha'=\alpha/(2+\alpha)$ if $\mathrm{Mod}(\X)< \mathrm{Mod}(\Y)$.

In \cite{arxivkalaj} the author extended the main result in \cite{annalen} and proved the following extension of the Kellogg theorem.
Every diffeomorphic minimiser of Dirichlet energy of Sobolev mappings between doubly connected Riemanian surfaces  $(\X,\sigma)$ and $(\Y,\rho)$ having $\mathscr{C}^{n,\alpha}$ boundary, $0<\alpha<1$, is $\mathscr{C}^{n,\alpha}$ up to the boundary, provided the metric $\rho$ is smooth enough. Here $n$ is a positive integer. It is crucial that, every diffeomorphic minimizer of Dirichlet energy is a harmonic mapping with a very special Hopf differential and this fact is used in the proof.

\subsection{Harmonic mappings between Riemann surfaces} Let $M=(\X, \sigma)$ and $N=(\Y,\wp)$ be
Riemann surfaces with metrics $\sigma$ and $\wp$, respectively. If
a mapping $f:M\to N,$ is $C^2$, then $f$ is said to be harmonic (to
avoid the confusion we will sometimes say $\wp$-harmonic) if

\begin{equation}\label{el}
f_{z\overline z}+{(\log \wp^2)}_w\circ f f_z\,f_{\bar z}=0,
\end{equation}
where $z$ and $w$ are the local parameters on $M$ and $N$
respectively. Also $f$ satisfies \eqref{el} if and only if its Hopf
differential
\begin{equation}\label{anal}
\mathrm{Hopf}(f)=\wp^2 \circ f f_z\overline{f_{\bar z}}
\end{equation} is a
holomorphic quadratic differential on $M$. Let $$|\partial
f|^2:=\frac{\wp^2(f(z))}{\sigma^2(z)}\left|\frac{\partial
f}{\partial z}\right|^2\text{ and }|\bar \partial
f|^2:=\frac{\wp^2(f(z))}{\sigma^2(z)}\left|\frac{\partial
f}{\partial \bar z}\right|^2$$ where $\frac{\partial f}{\partial z}$
and $\frac{\partial f}{\partial \bar z}$ are standard complex
partial derivatives. The $\wp-$Jacobian is defined by
$$J(f):=|\partial f|^2-|\partial \bar f|^2.$$ If $u$ is sense preserving,
then the $\wp-$Jacobian is positive. The Hilbert-Schmidt norm of
differential $df$ is the square root of the energy $e(f)$ and is
defined by
\begin{equation}\label{hil}|df|=\sqrt{2|\partial f|^2+2|\partial\bar
f|^2}.\end{equation}  For $g:M \mapsto N$ the $\wp-$
\emph{Dirichlet energy} is defined by
\begin{equation}\label{ener1} \E^\wp[g]=\int_{M} |dg|^2 dV_\sigma,
\end{equation}
where $\partial g$, and $\bar \partial g$ are the partial
derivatives taken with respect to the metrics $\varrho$ and
$\sigma$, and $dV_\sigma$, is the volume element on $(M,\sigma)$,
which in local coordinates takes the form $\sigma^2(w)du\wedge dv,$
$w=u+iv$. Assume that energy integral of $f$ is bounded. Then a
stationary point $f$ of the corresponding functional where the
homotopy class of $f$ is the range of this functional is a harmonic
mapping. The converse is not true. More precisely there exists a
harmonic mapping which is not stationary. The literature is not
always clear on this point because for many authors, a harmonic
mapping is a stationary point of the energy integral. For the last
definition and some important properties of harmonic maps see
\cite{jost}. It follows from the definition that, if $a$ is conformal
and $f$ is harmonic, then $f\circ a$ is harmonic. Moreover if $b$ is
conformal, $b\circ f$ is also harmonic but with respect to
(possibly) an another metric $\wp_1$.

 Moreover if $N$ and $M$ are double connected
Riemann surfaces with non vanishing metrics $\sigma$ and $\wp$,
then by \cite[Theorem~3.1]{jost}, there exist conformal mappings
$X\colon \X \onto M$ and $X^*\colon \Y\onto N$ between
double connected plane domains $\X$ and $\Y$ and Riemann
surfaces $M$ and $N$ respectively.

Notice that the harmonicity neither Dirichlet energy do not depend
on metric $\sigma$ on domain so we will assume from now on
$\sigma(z)\equiv 1$. This is why throughout this paper
$M=(\X,\mathbf{1})$ and $N=(\Y,\wp)$ will be doubly
connected domains in the complex plane $\C$ (possibly unbounded),
where $\mathbf{1}$ is the Euclidean metric. Moreover $\wp$ is a
nonvanishing smooth metric defined in $\Y$ with bounded Gauss
curvature $\mathcal{K}$ where
\begin{equation}\label{gaus}\mathcal{K}(z)=-\frac{\Delta \log \wp(z)}{\wp^2(z)},\end{equation} (we
put $\kappa:=\sup_{z\in \Y} |\mathcal{K}(z)|<\infty$) and with
finite area defined by
$$\mathcal{A}(\wp)=\int_{\Y}\wp^2(w) du dv, \ \ w=u+iv.$$

We call a metric $\rho$ an admissible one if there  is a constant $C_\wp>0$ so that \begin{equation}\label{pp}{|\nabla \wp(w)|}\le C_\wp{\wp(w)}, \ \ \ w\in\Y \ \ i.e. \ \ \nabla \log \wp\in L^\infty(\Y)\end{equation} which means that $\wp$ is so-called approximately analytic function (c.f. \cite{EH}).

Assume that the domain of $\wp$ is the unit disk $\mathbf{D}:=\{z: |z|<1\}\subset\mathbf{C}$.
From \eqref{pp} and boundedness of $\wp$, it follows that it is Lipschitz, and so it is continuous up to the boundary.  Again by using \eqref{pp}, the function $\psi(t)=\wp(te^{i\alpha})$, $0<t<1$, $\alpha\in[0,2\pi]$ satisfies the differential inequalities $-C_\wp\le \partial_t \log \psi(t)\le C_\wp$, which by integrating in $[0,t]$ imply that $\psi(0)e^{-C_\wp t}\le \psi(t)\le \psi(0)e^{C_\wp t}$. Therefore under the above
conditions  there holds the double inequality \begin{equation}\label{double}0<\wp(0)e^{-C_\wp} \le {\wp(w)}\le \wp(0)e^{C_\wp}<\infty, \ \ w\in\mathbf{D}.\end{equation}
A similar inequality to \eqref{double} can be proved for $\Y$ instead of $\mathbf{D}$.
  The Euclidean metric ($\wp\equiv 1$) is an admissible metric. The Riemannian metric defined by $\wp(w)={1}/{(1+|w|^2)^2}$ is admissible as well.
The Hyperbolic metric $h(w)={1}/{(1-|w|^2)^2}$ is not an admissible metric on the unit disk neither on the annuli $\A(r,1):= \{z:r<|z|<1\}$,
but it is admissible in $\A(r,R):=\{z: r<|z|<R\}$, where $0<r<R<1$.
We call such a metric
\emph{allowable} one (cf. \cite[P.~11]{Ahb}). If $\wp$ is a given
metric in $\Y$, we conventionally extend it to be equal to
$0$ in $\partial\Y$.
As we already pointed out, we will study the minimum of Dirichlet
integral of mappings between certain sets. We refer to introduction
of \cite{inve} and references therein for good setting of this
problem and some connection with the theory of nonlinear elasticity.
Notice first that a change of variables $w=f(z)$ in~\eqref{ener1}
yields
\begin{equation}\label{ener2}
{\E^\wp}[f] = 2\int_{\X} \wp^2(f(z))J_f(z)\, dz +
4\int_{\X}\wp^2(f(z)) \abs{f_{\bar z}}^2dz\ge 2
\mathcal{A}(\wp)
\end{equation}
where $J_f$ is the Jacobian determinant and $\mathcal{A}(\wp)$ is
the area of $\Y$ and $dz:=dx\wedge dy$ is the area element
w.r. to Lebesgue measure on the complex plane. A conformal mapping
of $f:\X\onto\Y$; that is, a homeomorphic solution of the
Cauchy-Riemann system $ f_{\bar z}=0$, would be an obvious choice
for the minimizer of~\eqref{ener2}. For arbitrary multiply connected
domains there is no such mapping.
%

Any energy-minimal diffeomorphism satisfies Euler-Laplace's
equation, since one can perform first variations while preserving
the diffeomorphism property. However, in the case of Euclidean
metric $\wp\equiv 1$, the existence of a harmonic diffeomorphism
between certain sets does not imply the existence of an
energy-minimal one, see \cite[Example~9.1]{inve}. Example~9.1 in
\cite{inve} has been constructed with help of affine self-mappings
of the complex plane. For a general metric $\wp$, affine
transformations are not harmonic, thus we cannot produce a similar
example.

\subsection{Statement of results}
The main results of this paper is the following
extension of the main result in \cite{geom}.
\begin{theorem} \label{mainexist}
Suppose that $\X$ and $\Y$ are doubly connected domains in
$\C$ with $\mathscr{C}^{1,\alpha}$ boundaries. Assume that $\D^{\wp}(\X,\Y)$ is the class of deformations as in  definition  \eqref{defdef} below.  Let $\wp$ be an
admissible metric in $\Y$. Then there exists a deformation  $w$
that minimizes $\wp-$energy throughout the class of deformations $\D^{\wp}(\X,\Y)$, and it is Lipschitz continuous up to the boundary of $\X$.
\end{theorem}

\begin{remark}
In contrast to the main result obtained in \cite{geom}, in Theorem~\ref{mainexist} the minimizer is not necessarily a diffeomorphism, and such a minimizer exists almost always except in some degenerate cases (see Proposition~\ref{attain} below).
\end{remark}
The following example is taken from \cite{calculus} for general radial metric. For a corresponing example for Euclidean metric see \cite{inve}.
\begin{example}
Assume that $\wp(w)=\rho(|w|)$ is a radial metric defined in the annulus $\Omega^*=\A(1,R)$. Assume that $\Omega=\A(1,r)$.
Chose  \begin{equation*}r<
\exp\left(\int_{1}^{R}
\frac{\rho(y)dy}{\sqrt{y^2\rho^2(y)-R^2\rho^2(R)
}}\right).\end{equation*} Let
$$r_\diamond=\exp\left(\int_{1}^{R}
\frac{\rho(y)dy}{\sqrt{y^2\rho^2(y)-R^2\rho^2(R) }}\right).$$ Then
$r<r_\diamond<1$.

Further for $\gamma=-\delta^2\rho^2(\delta)$,
we have well defined function
\begin{equation}\label{qu}q_{\diamond}(s)=\exp\left(\int_{\sigma}^s
\frac{dy}{\sqrt{y^2-\delta^2\rho^2(\delta)\varrho^2 }}\right),
\delta\le s\le \sigma.\end{equation}

Then the infimum
$\EE^\rho(\Omega,\Omega^*)$ is realized by a non-injective
deformation $h \colon \Omega \onto \Omega^\ast$
\[h(z)=\begin{cases} R\frac{z}{\abs{z}} & \mbox{for } r<\abs{z} \le r_\diamond\\
h_{\diamond}(z)& \mbox{ for } r_\diamond\le \abs{z} <1
\end{cases}\] where $h_{\diamond}(z)=p_{\diamond}(|z|)e^{it}=(q_{\diamond})^{-1}(|z|)e^{it}$, $z=|z|e^{it}$ and  $q_{\diamond}$ is defined in \eqref{qu} below

 Here the radial projection $z\mapsto R z/\abs{z}$
hammers $A(r,r_\diamond)$ onto the circle $|z|=R$ while the critical
$\rho-$Nitsche mapping $h_{\diamond}$ takes $A(r_\diamond,1)$
homeomorphically onto $\Omega^\ast$.
\end{example}
\section{Background and preliminaries}\label{stasec}

A homeomorphism of a planar domain is either sense-preserving or
sense-reversing. For  homeomorphisms of the Sobolev class
$\mathscr{W}^{1,1}_{\rm loc}(\X, \Y)$ this implies that the
Jacobian determinant does not change sign: it is either nonnegative
or nonpositive at almost every point~\cite[Theorem 3.3.4]{AIMb}. The homeomorphisms considered in this paper are
sense-preserving.

Let $\X$ and $\Y$ be domains in $\C$. Let $a\in \Y$
and $b\in \partial \Y$. We define
$$\dist_\wp(a,b):=\inf_\gamma \int_{\gamma}\wp(w)|dw|,$$ where $\gamma$ ranges over all
rectifiable Jordan arcs connecting $a$ and $b$ within $\Y$ if
the set of such Jordan arcs is not empty (otherwise we
conventionally put $\dist_\wp(a,b)=\infty$). To every mapping $f
\colon \X \to \overline{\Y}$ we associate a boundary
distance function
$$\delta^\wp_f(z)=\dist_\wp (f(z), \partial
\Y)=\inf_{b\in \partial \Y}\dist_\wp(f(z),b)$$ which
is set to $0$ on the boundary of $\X$.

We now adapt for our purpose the concepts of $c\delta-$uniform
convergence and of deformation defined for Euclidean metric and
bounded domains in \cite{inve}.
\begin{definition}
A sequence of continuous mappings $h_j\colon \X\to {\Y}$
is said to converge \emph{$c\delta$-uniformly} to $h\colon \X\to
\overline{\Y}$ if
\begin{enumerate}
\item $h_j\to h$ uniformly on compact subsets of $\X$ and
\item $\delta^\wp_{h_j} \to  \delta^\wp_h$ uniformly on $\overline{\X}$.
\end{enumerate}
We designate it as $h_j\cto h$. Concerning the item (1) we need to
notice that the Euclidean metric and the nonvanishing smooth metric
$\wp$ are equivalent on compacts of $\Y$.
\end{definition}

\begin{definition}\label{defdef}  A mapping $h\colon \X\to\overline{\Y}$
is called a \emph{$\wp$ deformation} if
\begin{enumerate}
\item\label{dd1}   $h\in W_{loc}^{1,2}$ and $|dh|\in L^2$ (which we write shortly $h\in W_l^{1,2}$);
\item\label{dd2} The Jacobian $J_h:=\det Dh$ is nonnegative  a.e. in $\X$;
\item\label{dd2p} $\int_\X \wp^2(h(z))J_h  \le \mathcal{A}(\wp)$;
\item\label{dd3} there exist sense-preserving  diffeomorphisms  $h_j\colon \X\onto\Y$, called an \emph{approximating sequence},  such that $h_j\cto h$  on $\X$.
\end{enumerate}
 The set of $\wp$ deformations
$h\colon \X \to \overline{\Y}$ is denoted by
$\D^\wp(\X,\Y)$.
\end{definition}
Notice first that the condition (1) of the previous definition in
the case of bounded domains $\X$ and $\Y$ can be replaced
by  $h\in W^{1,2}$ see \cite[Lemma~2.3.]{calculus}.

\begin{lemma}\label{exc}
Let $\X$ and $\Y$ be bounded doubly connected planar
domains. Assume that the boundary components of $\X$ do not
degenerate into points. If a sequence $\{h_j\}\subset
\D^\wp(\X,\Y)$ converges weakly in $W^{1,2}$, then its
limit belongs to $\D^\wp(\X,\Y)$. In particular, every such limit function has continuous extension to the boundary and the boundary is mapped into the boundary.
\end{lemma}
Now we formulate the following existence result proved in \cite{calculus}.
\begin{proposition}\label{attain}\cite[Lemma~2.15]{calculus}
Let $\X$ and $\Y$ be bounded doubly connected planar
domains. Assume that the boundary components of $\X$ do not
degenerate into points.  There exists $h\in \D^\wp
(\X, \Y)$ such that ${\E^\wp}[h]= \EE^\wp (\X,
\Y)$, where $$\EE^\wp (\X,\Y):=\inf\{\E^\wp[h]: h \in \D^\wp(\X,\Y)\}.$$
\end{proposition}
Assume that $\Ho_\wp(\X,\Y)$ is the class of $\mathscr{W}^{1,2}$ homeomorphic mappings between $\X$ and $\Y$ so that $$\int_{\X} \wp^2(f(z))(|f_z|^2+|f_{\bar z}|^2)dxdy +\int_{\X} \wp^2(f(z))|f(z)|^2 dxdy<\infty,$$ and let $\overline{\Ho_\wp(\X,\Y)}$ be its closure in  $\mathscr{W}^{1,2}$. Further we have that $\overline{\Ho_\wp(\X,\Y)}\subset
\D^\wp(\X,\Y)$.

Now as in \cite[Theorem~1.1.]{arma0} is can be proved the following modification of the previous proposition
\begin{proposition}\label{attain}\cite[Lemma~2.15]{calculus}
Let $\X$ and $\Y$ be bounded doubly connected planar
domains. Assume that the boundary components of $\X$ do not
degenerate into points.  There exists $h\in \overline{\Ho_\wp
(\X, \Y)}$ such that $${\E^\wp}[h]= \EE^\wp (\X,
\Y)=\inf\{\E^\wp[h]: h \in \Ho_\wp
(\X, \Y)\}.$$
\end{proposition}

\subsection{Stationary mappings}\label{hopsec}

We call a mapping $h\in\overline{\Ho_\rho(\X,\Y)}$
\emph{stationary} if
\begin{equation}\label{stat}
\frac{d}{dt}\bigg|_{t=0}{\E^\wp}[h\circ \phi_t^{-1}]=0
\end{equation}
for every family of diffeomorphisms $t\to \phi_t\colon
\X\to\X$ which depend smoothly on the parameter $t\in\mathbb
R$ and satisfy
$\phi_0=\id$. The latter mean that the mapping $\X\times [0,\epsilon_0]\ni (t,z)\to \phi_t(z)\in \X $ is a smooth mapping for some $\epsilon_0>0$.
We now have.
\begin{lemma}\label{ctheory}\cite{calculus}
Let $\X=A(r,R)$ be a circular annulus, $0<r<R<\infty$, and
$\Y$ a doubly connected domain. If $h\in
\mathscr{W}^\wp(\X,\Y)$ is a stationary mapping, then
\begin{equation}\label{hopf1}\wp^2(h(z))
h_z\overline{h_{\bar z}} \equiv \frac{c}{z^2}\qquad \text{in }\X
\end{equation}
where $c\in\R$ is a constant.
\end{lemma}
Notice that the corresponding lemma in \cite{calculus} is for deformations, but the proof works as well for the class $ \overline{\Ho^\wp(\X,\Y)}\subset \D^\wp(\X,\Y)$.

\section{Proof of the main result}
We need the following important result concerning the local Lipschitz character of certain mappings proved by Iwaniec, Kovalev and  Onninen  in \cite{duke}.
\begin{proposition}\label{LipHopf2}
Let $h\in \mathscr W^{1,2}(\X)$ be  a mapping with nonnegative Jacobian.
Suppose that the Hopf product $h_z\,\overline{h_{\bar z}}$ is bounded and H\"{o}lder continuous.
Then $h$ is locally Lipschitz but not necessarily $\mathscr C^1$-smooth.
\end{proposition}
We need the also the following three results.
\begin{lemma}\label{popi}
Every sense-preserving solution of Hopf equation so that $\mathrm{Hopf}(f)$ is bounded that maps $\A(1,R)$ into $\Y$, mapping the inner/outer boundary to inner/outer boundary is $(K,K')$ quasiconformal, where $$K=1 \ \ \text{and}\ \ K'=\sup_{z\in \A(1,R)}\frac{|c|}{|z|^2\wp^2(f)}.$$
\end{lemma}
\begin{proof} We have
\[\begin{split}|Df(z)|^2&=|f_z|^2+|f_{\bar z}|^2 \\&\le (|f_z|^2-|f_{\bar z}|^2)+2|f_{\bar z}|^2\\&\le (|f_z|^2-|f_{\bar z}|^2)+2|f_{\bar z} f_{ z}|.\end{split}\]

Since $$|f_{\bar z} f_{ z}|=\frac{\mathrm{Hopf}(f)}{|\wp(f(z))|^2}=\frac{|c|}{|z|^2\wp^2(f)}\le K'.$$ This implies the claim.
\end{proof}



\begin{proposition}\cite[Theorem~12.3]{gt}\label{gilbarg}
Let $K>1, K'\ge 0$ and assume that $f:\X\to \mathbf{C}$ is a $(K,K')-$quasiconformal mapping so that $|f(z)|\le M$, $z\in \X$ and assume that $\X'\Subset \X$ and let $d=\mathrm{dist}(\X',\partial\X)$. Then there is a constant $C=C(K)$ so that $$|f(z)-f(z')|\le C(K)(M+d\sqrt{K'})|z-z'|^\beta,\ z,z'\in\X'$$ where
$$\beta=K-\sqrt{K^2-1}.$$
If $K=1$, then the above theorem can be formulated for $K_1$ instead of $K$, where $K_1>1$ is an arbitrary constant. For example for $K_1=5/4$, for which we get $\beta =K_1-\sqrt{K_1^2-1}= 1/2$.
\end{proposition}
We also need the following lemma
\begin{lemma}\label{pali}\cite{calculus}
Assume that $\X$ and $\Y$ are doubly connected domains,
and assume that $a\colon \X \onto \A(1,R)$ and $b\colon
\A(1,\rho) \onto \Y$ are univalent conformal mappings and
define $\wp_1(w)={\wp(b(w))}{|b'(w)|}$, $w\in \A(1,\rho)$. Then
\begin{enumerate}[\ \ \ \ (a)]
\item $\E^{\wp}[b\circ
f\circ a]=\E^{\wp_1}[f]$ provided that one of the two sides exist.
    \item $b\circ f\circ a \in D^\wp(\X,\Y)$ if and only if
$f\in D^{\wp_1}(\A(1,R),\A(1,\rho)))$.
    \item For Gauss curvature we have $\mathcal{K}_\wp(b(w))=\mathcal{K}_{\wp_1}(w)$.
\item  $\wp$ is an allowable metric if and only if
$\wp_1$ is an allowable metric.
    \item
$W_l^{1,2}(\A(1,R),\A(1,\rho))=W^{1,2}(\A(1,R),\A(1,\rho)).$
\item $b\circ f\circ a$ is $\wp$-harmonic if and only if $f$ is
$\wp_1-$harmonic.
\end{enumerate}

\end{lemma}
\begin{proof}[Proof of Theorem~\ref{mainexist}] The existence part has been discussed in Proposition~\ref{attain}. We will apply a self-improving argument.
In view of Lemma~\ref{pali}, we can assume that $\X=\A(1,R)$ for a constant $R>1$.
Assume that $f: \mathbb{A}(1,R) \to \mathbb{Y},$ so that $$\wp^2(f(z))f_z \bar f_z = \frac{c}{z^2}.$$
Let $\Psi: \mathbb{Y} \to \mathbb{A}(1, \rho)$ be a conformal diffeomorphism and let $\Phi=\Psi^{-1}$.

Then $F(z) = \Psi \circ f: \mathbb{A}(1,R)\to \mathbb{A}(1,\rho)$ is a minimizer of $\wp_1-$ energy, between $\A(1,R)$ and $\A(1,\rho)$, where $\wp_1(\zeta)={\wp(\Phi(\zeta))}{|\Phi'(\zeta)|}$
and \begin{equation}\label{lindje}\mathrm{Hopf}(F)(z)=\wp_1^2(F(z))F_z  \bar F_z = \frac{c}{z^2}.\end{equation}
Further for
$$\tilde F(z) = \left\{
                  \begin{array}{ll}
                    F(z), & \hbox{$1< |z|\le R$;} \\
                    \rho^2 /\overline{F(R^2/\bar z)}, & \hbox{$R\le |z|< R^2$}
\\ 1/\overline{F(1/\bar z)}, &\hbox{$1/R\le |z|\le 1$},
                  \end{array}
                \right.$$
let $$\tilde \wp(w)=\left\{
                   \begin{array}{ll}
                     \wp_1(w), & \hbox{$1\le |w|\le \rho $;} \\
                     \wp_1(\rho^2/\bar w), & \hbox{$\rho<|w|<\rho^2$,}\\
                     \wp_1(1/\bar w), & \hbox{$1/\rho<|w|\le 1$.}

                   \end{array}
                 \right.$$
Observe that $\tilde F$ is continuous. Namely, if
$F(Re^{it})=\rho e^{is},$ then $$\rho^2/\overline{F(R^2/(Re^{-it}))}=\rho e^{is}.$$
Thus $\tilde F: \mathbb{A}(1/R,R^2) \to \mathbb{A}(1/\rho,\rho^2)$ is continuous and belongs to the same class as $F$, which mean it is in $\mathscr{W}^{1,2}$.
Then by direct  calculation we get
$$\mathrm{Hopf}(\tilde F) =\tilde\wp^2(\tilde F(z))\tilde{F}_z\overline{\tilde{F}}_z=\left\{
                      \begin{array}{ll}
                         \frac{c}{z^2}, & \hbox{$1< |z|\le R$;} \\
                        \frac{\rho^4}{|F(R^2/\bar z)|^4} \frac{c}{z^2}, & \hbox{$R\le |z|< R^2$.}
 \\
                        \frac{1}{|F(1/\bar z)|^4} \frac{c}{z^2}, & \hbox{$1/R\le |z|< 1$.}
                      \end{array}
                    \right.
$$
Let us demonstrate for example $$\mathrm{Hopf}(\tilde F)=\frac{1}{|F(1/\bar z)|^4} \frac{c}{z^2}$$ for  $1/R\le |z|< 1$.

We have $$\tilde F_z(z) = \frac{d}{dz}\left(\frac{1}{\overline{F}(\frac{1}{\bar z})}\right)=\frac{1}{\overline{F^2}(\frac{1}{\bar z})}\overline{F_z\left(\frac{1}{\bar z}\right)} \frac{1}{z^2}$$
and
$$\overline{\tilde F}_z(z) = \frac{d}{dz}\left(\frac{1}{{F}(\frac{1}{\bar z})}\right)=\frac{1}{{F^2}(\frac{1}{\bar z})}\overline{\overline{F}_z\left(\frac{1}{\bar z}\right)} \frac{1}{z^2}.$$

Then from \eqref{lindje} we get $$\tilde\wp^2 (\tilde F(z))\tilde F_z(z) \overline{\tilde F}_z(z)=\frac{1}{|F(\frac{1}{\bar z})|^4}\frac{cz^2}{z^4}=\frac{1}{|F(\frac{1}{\bar z})|^4}\frac{c}{z^2}.$$

From Kellogg theorem for conformal mappings we know that $\Phi'$ is $\mathscr{C}^{\alpha}$ up to the boundary of $\Y$. This implies in particular that the function $$\tilde F_z \overline{\tilde F}_z=\frac{\mathrm{Hopf}(\tilde F)}{\tilde \wp^2(\tilde F(z))}$$ is bounded in $\mathbb{A}(1/R,R^2)$. From Lemma~\ref{popi}, in view of Lemma~\ref{exc}, we obtain that $\tilde F$ is $(1,\tilde K)$-quasiconformal. Now Proposition~\ref{gilbarg} implies that $F$ is $\frac{1}{2}-$H\"older continuous in $\A(1,R).$ More precisely $$|\tilde F(z) - \tilde F(z')|\le C(1)\left(\rho^2+\frac{R^2-1}{R^2}\sqrt{\tilde K}\right)|z-z'|^{1/2}.$$

So

$$|F(z) - F(z')|\le C(1)\left(\rho^2+\frac{R^2-1}{R^2}\sqrt{\tilde K}\right)|z-z'|^{1/2}.$$

This implies that $\tilde F_z \overline{\tilde F}_z$ is $\alpha\cdot 1/2-$H\"older continuous as a composition of two mappings which are $\alpha-$ and $1/2-$H\"older continuous respectively in $\A(1/R, R^2)$. Now Theorem~\ref{LipHopf2} implies that $\tilde F$ is locally Lipschitz. In particular $F$ is Lipschitz in $\overline{\A(1,R)}\Subset \A(1/R, R^2)$. Therefore $f:\A(1,R)\to \Y$ is Lipschitz continuous. Now Lemma~\ref{pali} and Kellogg theorem for conformal mappings imply the claim. This finishes the proof of the main result.
\end{proof}

\end{document}